\documentclass[11pt]{amsart}
\usepackage{amssymb,latexsym}
\usepackage{enumerate}
\usepackage[T1]{fontenc}
\usepackage{pxfonts}
\usepackage{mathrsfs}
\usepackage{amsmath}
\usepackage{yhmath}
\makeatletter
\@namedef{subjclassname@2010}{%
  \textup{2010} Mathematics Subject Classification}
\makeatother
\makeindex

\newtheorem{theorem}{Theorem}[subsection]
\newtheorem{cor}[theorem]{Corollary}

\theoremstyle{definition}
\newtheorem{df}[theorem]{Definition}
\newtheorem{rem}[theorem]{Remark}

\numberwithin{theorem}{section}

\newcommand{\mb}{\mathbb}
\newcommand{\mc}{\mathcal}
\newcommand{\mf}{\mathfrak}

\newcommand{\s}{\subset}
\newcommand{\bs}{\boldsymbol}

\begin{document}
\title[On behaviour...]{On behaviour of holomorphically contractible systems under non-monotonic sequences of sets}

\author{Arkadiusz Lewandowski}

\address{Institute of Mathematics\\ Faculty of Mathematics and Computer Science\\ Jagiellonian University\\ {\L}ojasiewicza 6,
30-348 Kraków, Poland}

\address{Science Institute\\ School of Engineering and Natural Sciences\\ University of Iceland\\ Dunhaga 3,
107 Reykjavík, Iceland}
\email{arkadiuslewandowski@gmail.com; Arkadiusz.Lewandowski@im.uj.edu.pl; arkadius@hi.is}
\thanks{}
\date{}

\begin{abstract}
The new results concerning the continuity of holomorphically contractible systems treated as set functions with respect to non-monotonic sequences of sets are given. In particular, continuity properties of Kobayashi and Carathéodory pseudodistances, as well as Lempert and Green functions with respect to sequences of domains converging in Hausdorff metric are delivered. 
\end{abstract}

\maketitle

\section{Introduction}
It is known that both Carathéodory and Kobayashi pseudodistances depend continuously on increasing and decreasing sequences of domains (in the latter case, adding some regularity assumptions on limiting domain; cf. [\ref{KCOIDODD}] and references therein). The pseudodistances mentioned above are particular examples of wider class of holomorphically contractible systems, i.e. systems of functions 
$$
\bs{d}_D:D\times D\rightarrow[0,+\infty),
$$
$D$ running through all domains in all $\mb{C}^n$'s, such that $\bs{d}_{\mb{D}}$ is forced to be $\bs{p}$, the hyperbolic distance on $\mb{D}$, the unit disc on the plane and all holomorphic mappings are contractions with respect to the system $(\bs{d}_D)$ (cf. Definition \ref{holocontr}). The question about the behaviour of holomorphically contractible systems under not necessarily monotonic sequences of sets seems to be natural and important. In the present note, inspired by [\ref{BKS}], we shall give a very general result stating the continuity of holomorphically contractible systems under the sequences of domains convergent with respect to Hausdorff distance (for two nonempty bounded sets $A,B$ it is defined as 
$$
\mf{H}(A,B):=\inf\{\delta>0: A\s B^{(\delta)} \text{\ and\ } B\s A^{(\delta)}\},
$$
where for a set $S$ and a positive number $\varepsilon$, the set $S^{(\varepsilon)}:=\bigcup_{s\in S}\mb{B}(s,\varepsilon)$ is the \emph{$\varepsilon$-envelope of $S$}; $\mb{B}(x,r)$ denotes the open Euclidean ball of center $x$ and radius $r$). Namely, our main result reads as follows:
\begin{theorem}\label{hd}
Let $(\bs{d}_D)$ be a holomorphically contractible system and let $D\s\mb{C}^m$ be a bounded domain. Assume that there exist two sequences $(I_n)_{n\in\mb{N}}, (E_n)_{n\in\mb{N}}$ of domains such that
$$
E_{n+1}\s\s E_n, n\in\mb{N}, \bigcap_{n\in\mb{N}}E_n=\overline{D}, I_n\s\s I_{n+1}, n\in\mb{N}, \bigcup_{n\in\mb{N}}I_n=D
$$
and such that for each $z,w\in D$ there is 
$$
\lim_{n\rightarrow \infty}\bs{d}_{E_n}(z,w)=\lim_{n\rightarrow\infty}\bs{d}_{I_n}(z,w)=\bs{d}_D(z,w).
$$
Let $(D_n)_{n\in\mb{N}}$ be a sequence of bounded domains in $\mb{C}^m$ such that 
$$
\lim_{n\rightarrow\infty} \mf{H}(\overline{D}_n,\overline D)=0
$$
and such that for each compact $K\s D$ there exists an $n_0\in\mb{N}$ such that for any $n\geq n_0, K\s D_n$. Then for any $z,w\in D$
$$
\lim_{n\rightarrow\infty}\bs{d}_{D_n}(z,w)=\bs{d}_D(z,w).
$$
\end{theorem}
In particular, we get the results in this spirit for Carathéodory and Kobayashi pseudodistances as well as for Green and Lempert functions (cf. Corollaries \ref{corkoblem}, \ref{corcara}, and \ref{corgre}). We believe they are interesting in their own.\\ 
\indent In  [\ref{BKS}] all the results are settled in the context of complex Banach spaces, yet under strong assumption about the convexity of the approximating domains together with the limiting one. Our results are free from this restrictive assumption.\\
\indent In Section 2 we give the formal definition of holomorphically contractible system and both list and prove the corollaries form Theorem \ref{hd}, while the proof of the main result itself comes in Section 3.\\
\indent In what follows, $\mc{O}(D,G)$ stands for the family of all holomorphic mappings between open sets $D,G$ and $\mc{PSH}(D)$ abbreviates the family of all plurisubharmonic functions on open set $D$.
\section{Holomorphically contractible systems}
Let us start with the precise definition of holomorphically contractible system.
\begin{df}[Cf. {[\ref{JPIDAMICA}]}, Section 4.1]\label{holocontr}
A family $(\bs{d}_D)$ of functions
$$
\bs{d}_D:D\times D\rightarrow[0,+\infty),
$$
where $D$ runs over all domains in $\mb{C}^n$ with arbitrary $n$, is called a \emph{holomorphically contractible system} if the following two conditions are satisfied:
\begin{enumerate}
\item $\bs{d}_{\mb{D}}=\bs{p}$
\item for any two domains $D\s\mb{C}^n,G\s\mb{C}^m$ and any mapping $f\in\mc{O}(D,G)$ there is
$$
\bs{d}_G(f(z),f(w))\leq\bs{d}_D(z,w), \quad z,w\in D.
$$
\end{enumerate}
\end{df}
\begin{rem}
If in the above definition we replace $\bs{p}$ by $\bs{m}$, the Möbius distance on $\mb{D}$, then we speak of \emph{$\bs{m}$-contractible system}. This distinction is however somewhat artificial, since having $(\bs{d}_D)$, a holomorphically contractible system, we may define $\bs{d}_D^{*}:=\tanh \bs{d}_D$ and then the operator sending $(\bs{d}_D)$ to $(\bs{d}_D^{*})$ is a bijection between the class of contractible systems and the class of $\bs{m}$-contractible systems (see {[\ref{JPIDAMICA}]}, Section 4.1).
\end{rem}
The most important examples of holomorphically contractible systems are the following:
\begin{enumerate}
\item \emph{Carathéodory pseudodistance}:
$$
\bs{c}_D(z,w):=\sup\{\bs{p}(0,f(w)): f\in\mc{O}(D,\mb{D}), f(z)=0\},\quad z,w\in D.
$$
\item \emph{Lempert function}:
\begin{multline*}
\bs{l}_D(z,w):=\\ \inf\{\bs{p}(0,\lambda):\lambda\in{\mb{D}: \text{\ there\ exists\ a\ }\varphi\in\mc{O}(\mb{D},D)}:\varphi(0)=z,\varphi(\lambda)=w\},\quad z,w\in D.
\end{multline*}
\item \emph{Kobayashi pseudodistance}:
\begin{multline*}
\bs{k}_D(z,w):=\\ \inf\big\{\sum_{j=1}^N\bs{l}_D(z_{j-1},z_j):N\in\mb{N},z_1,\ldots,z_N\in D,z_0=z, Z_N=w\big\}, \quad z,w\in D.
\end{multline*}
\item \emph{Green function}:
\begin{multline*}
\bs{g}_D(z,w):=\sup\{u(z):u:D\rightarrow[0,1): \log u\in\mc{PSH}(D),\text{\ there\ exist\ }M,r>0:\\u(z)\leq M\|\zeta-w\|,\zeta\in\mb{B}(w,r)\s D\}
\end{multline*}
forms an example of $\bs{m}$-contractible system.
\end{enumerate}
Note that $(\bs{c}_D)$ and $(\bs{k}_D)$ are extremal holomorphically contractible systems of pseudodistances, i.e. if $(\bs{d}_D)$ is any holomorphically contractible system of pseudodistances, it verifies the inequalities
$$
\bs{c}_D\leq\bs{d}_D\leq\bs{k}_D
$$ 
for all domains $D$. Similarly, if $(\bs{d}_D)$ is any holomorphically contractible system of functions, then
$$
\bs{c}_D\leq\bs{d}_D\leq\bs{l}_D
$$
for all domains $D$ (see [\ref{JPIDAMICA}], Section 4.1).\\
\indent Having Theorem \ref{hd} we may settle the continuity results for particular objects. 
\begin{cor}\label{corkoblem}
Let $D\s\mb{C}^m$ be a bounded taut domain with boundary of class $\mc{C}^{1,1}$. Let $(D_n)_{n\in\mb{N}}$ be a sequence of bounded domains in $\mb{C}^m$ such that 
$$
\lim_{n\rightarrow\infty}\mf{H}(\overline{D}_n,\overline D)=0
$$
Assume that for each compact $K\s D$ there exists an $n_0\in\mb{N}$ such that for any $n\geq n_0, K\s D_n$. Then for any $z,w\in D$
$$
\lim_{n\rightarrow\infty}\boldsymbol{k}_{D_n}(z,w)=\boldsymbol{k}_D(z,w)
$$
as well as
$$
\lim_{n\rightarrow\infty}\boldsymbol{l}_{D_n}(z,w)=\boldsymbol{l}_D(z,w)
$$
\end{cor}
\begin{proof}
Indeed, in virtue of the regularity assumption, one can take
$$
E_n=D^{(\frac{1}{N_0+n})}:=\{z\in\mb{C}^m:\text{\ dist}(z,D)<\frac{1}{N_0+n}\}
$$
and
$$
I_n=D^{(-\frac{1}{N_0+n})}:=\{z\in D:\text{\ dist}(z,\partial{D})>\frac{1}{N_0+n}\}
$$
with $N_0\in\mb{N}$ large enough and make use of the continuity of Kobayashi pseudodistance and Lempert function with respect to monotonic sequences of domains (see {[\ref{KCOIDODD}]} and references therein).
\end{proof}
\begin{cor}\label{corcara}
Let $D\s\mb{C}^m$ be a bounded strictly pseudoconvex domain with boundary of class $\mc{C}^{2}$. Let $(D_n)_{n\in\mb{N}}$ be a sequence of bounded domains in $\mb{C}^m$ such that 
$$
\lim_{n\rightarrow\infty}\mf{H}(\overline{D}_n,\overline D)=0
$$
Assume that for each compact $K\s D$ there exists an $n_0\in\mb{N}$ such that for any $n\geq n_0, K\s D_n$. Then for any $z,w\in D$
$$
\lim_{n\rightarrow\infty}\boldsymbol{c}_{D_n}(z,w)=\boldsymbol{c}_D(z,w).
$$
\end{cor}
\begin{proof}
The proof goes along the same lines as the proof of Corollary \ref{corkoblem}.
\end{proof}
In the case of Green function, things go a little bit more complicated. Let us see the details.
\begin{df}\label{strhyp}
Let $D\s\mb{C}^m$ be a bounded domain.
\begin{enumerate}
\item $D$ is \emph{hyperconvex} if there exists a continuous and negative plurisubharmonic exhaustive function on $D$.
\item $D$ is \emph{strictly hyperconvex} if there exist a bounded domain $\Omega$ and a continuous function $\rho\in\mc{PSH}(\Omega)$ with values in $(-\infty, 1)$ such that $D=\{z\in\Omega:\rho(z)<0\}$, $\rho$ is exhaustive for $\Omega$, and the sublevel sets $\{z\in\Omega:\rho(z)<\alpha\}$ are connected for $\alpha\in[0,1]$. 
\end{enumerate}
\end{df}
One can observe that strictly hyperconvex domain is a hyperconvex domain with negative continuous exhaustive function that can be plurisubharmonically and continuously extended to some open neighbourhood of the closure of the domain. The examples of such domains are bounded strictly pseudoconvex domains with $\mc{C}^2$ boundary.  
\begin{cor}\label{corgre}
Let $D\s\mb{C}^m$ be a strictly hyperconvex domain. Let $\rho$ be as in Definition \ref{strhyp}. Assume that $D^k$ is a hyperconvex domain given by $\{z\in\Omega:\rho(z)<\frac{1}{k}\}, k\in\mb{N}$. Let $(D_n)_{n\in\mb{N}}$ be a sequence of bounded domains in $\mb{C}^m$ such that 
$$
\lim_{n\rightarrow\infty}\mf{H}(\overline{D}_n,\overline D)=0
$$
Assume that for each compact $K\s D$ there exists an $n_0\in\mb{N}$ such that for any $n\geq n_0, K\s D_n$. Then for any $z,w\in D$
$$
\lim_{n\rightarrow\infty}\boldsymbol{g}_{D_n}(z,w)=\boldsymbol{g}_D(z,w).
$$
\end{cor}
\begin{proof}
By [\ref{NTPGF}] we know that the Green function is continuous with respect to increasing sequences of domains. Therefore, $(I_n)_{n\in\mb{N}}$ may be chosen as some exhausting sequence of smoothly bounded strictly pseudoconvex relatively compact open subsets of $D$. Also, using results of [\ref{NVTTHM}], it is clear that the good candidate for the "exterior" sequence is $(E_n)_{n\in\mb{N}}:=(D^n)_{n\in\mb{N}}$. 
\end{proof}
\section{Proof of Theorem \ref{hd}}

\begin{proof}[Proof of Theorem \ref{hd}]
There exists an $m_{1}\in\mb{N}$ such that for $m\geq m_{1}$ we have
$$
I_1 \s\s D_m \s\s E_1.
$$
We may choose the smallest possible such an $m_{1}$. In what follows, we shall construct two sequences of sets, 
$(L_n)_{n\in\mb{N}}, (U_n)_{n\in\mb{N}}$, such that $L_n\s L_{n+1},n\in\mb{N}, \bigcup_{n=1}^{\infty} L_n=D, U_{n+1}\s\s U_n,n\in\mb{N}, \bigcap_{n=1}^{\infty} U_n=\overline{D}$ and
$$
L_n\s D_{m_{1}+n-1}\s U_n, n\in\mb{N}.
$$
Then for $n$ large enough, $z,w\in L_n$ and
$$
\boldsymbol{d}_{U_n}(z,w)\leq\boldsymbol{d}_{D_{m_{1}+n-1}}(z,w)\leq\boldsymbol{d}_{L_n}(z,w).
$$
Finally, letting $n\rightarrow\infty$ and using the assumptions concerning continuity of system $(\bs{d}_D)$ with respect to  monotonic sequences of domains $(I_n)_{n\in\mb{N}}, (E_n)_{n\in\mb{N}}$, we reach the conclusion of Theorem \ref{hd}. Let us pass to the construction.\\
Let $L_1:=I_1, U_1:=E_1$. We proceed as follows:\\
Choose the smallest $m_{2}\in\mb{N}$ such that for any $m\geq m_{2}$ we have
$$
I_2\s\s D_m \s\s E_2.
$$
There are two cases to be considered:\\
\emph{Case 1}. $m_{2}\in\{m_{1},m_{1}+1\}$. Then
$$
I_2\s\s D_{m_{2}} \s\s E_2
$$
and we put $L_2:=I_2, U_2:=E_2$.\\
\emph{Case 2}. $m_{2}=m_{1}+s$ with some $s\geq 2$. Then
$$
I_1\s\s D_l \s\s E_1, l=m_{1},\ldots,m_{1}+s,
$$
and so
$$
I_1\s\s \bigcup_{l=m_{1}+1}^{m_{1}+s-1}D_l \s\s E_1.
$$
We define $L_2=\ldots=L_s:=I_1, L_{s+1}:=I_2$. Further, as $U_2$ we choose a domain relatively compact in $E_1$, containing in its interior $\overline{E_2\cup\bigcup_{l=m_{1}+1}^{m_{1}+s-1}D_l}$. Inductively, for $k=2,\ldots,s$ a domain $U_k$ is chosen as a domain relatively compact in $U_{k-1}$, containing in its interior $\overline{E_2\cup\bigcup_{l=m_{1}+k-1}^{m_{1}+s-1}D_l}$. Finally, we put $U_{s+1}:=E_2$.\\
Suppose we have constructed domains $L_1\s \ldots\s L_r$ and $U_1\s\s \ldots\s\s U_r$ such that
$$
L_j\s D_{m_{1}+j-1}\s U_j, j=1,\ldots, r
$$
and $L_r=I_M, U_r=E_M, m_{1}+r-1=m_M$ with some $M\in\mb{N}$ We choose the smallest $m_{M+1}\in\mb{N}$ with
$$
I_{M+1}\s\s D_m \s\s E_{M+1}, m\geq m_{M+1}.
$$
Similarly as before, there are two cases to be considered:\\
\emph{Case 1}. $m_{M+1}\in\{m_M,m_M+1\}$. Then we put $L_{r+1}:=I_{M+1}, U_{r+1}:=E_{M+1}$.\\
\emph{Case 2}. $m_{M+1}=m_M+s$ with some $s\geq 2.$ Then we mimic the previously presented construction with necessary modifications.
\end{proof}


\begin{thebibliography}{HD}
\bibitem {}{M. Budzy{\'n}ska, T. Kuczumow, S. Reich}, \emph{Limiting behavior of the Kobayashi distance}, Taiwanese J. Math. 19 (2015), no. 2, 535--552. \label{BKS}
\bibitem {}{M. Jarnicki, P. Pflug}, \emph{Invariant Distances and Metrics in Complex Analysis, 2nd edition}, de Gruyter Expositions in Mathematrics 9, Walter de Gruyter 2014. \label{JPIDAMICA}
\bibitem {}M. Kobayashi, \emph{Convergence of Invariant Distances on decreasing domains}, Complex Variables 47 (2002), 155-165. \label{KCOIDODD}
\bibitem {}{S. Nivoche}, \emph{The pluricomplex Green function, capacitative notions, and approximation
problems in $\mb{C}^n$}, Indiana Univ. Math. J. 44 (2) (1995) 489-510. \label{NTPGF}
\bibitem {} {N. Van Trao, T. Hue Minh}, \emph{Remark on the Kobayashi hyperbolicity of complex spaces}, Acta Mathematica Vietnamica 34 (3) (2009), 375-387. \label{NVTTHM}
\end{thebibliography}
\end{document}